\documentclass[11pt]{amsart}
\usepackage{lmodern}
\usepackage[T1]{fontenc}
\usepackage{amssymb, amsmath, amsfonts}
\usepackage[raiselinks=false,colorlinks=true,citecolor=blue,urlcolor=blue,linkcolor=blue,bookmarksopen=true,pdftex]{hyperref}
\usepackage{graphicx}
\usepackage{setspace}
\usepackage[T1]{fontenc}
\usepackage[margin=1.0in]{geometry}
\def\scfig #1 #2 {\resizebox{#2}{!}{\includegraphics{#1}}}

\numberwithin{equation}{section}

\def\ker{\operatorname{ker}}

\theoremstyle{plain}
\newtheorem{theorem}[equation]{Theorem}

\newtheorem{lemma}[equation]{Lemma}

\theoremstyle{definition}

\newtheorem{definition}[equation]{Definition}

\newtheorem{remark}[equation]{Remark}

\newtheorem{discussion}[equation]{Discussion}

\usepackage{graphics}
\usepackage{color}

\begin{document}
 \title[The group of quasi-isometries of the real line]{On the center of the group of quasi-isometries of the real line}

\author{Prateep Chakraborty}
\address{Department of mathematics and Statistics, Indian Institute of Science Education and Research kolkata, Mohanpur- 741246, West Bengal, India.}
\email{chakraborty.prateep@gmail.com}

\thanks{The author is supported by an N.B.H.M. postdoctoral fellowship.}
\maketitle

\begin{abstract}
Let $QI(\mathbb{R})$ denote the group of all quasi-isometries $f:\mathbb{R}\to \mathbb{R}.$ Let $Q_+( \text{and}~ Q_-)$ denote the subgroup of $QI(\mathbb{R})$ consisting of elements which are identity near $-\infty$ (resp. $+\infty$).   We denote by $QI^+(\mathbb R)$ the index $2$ subgroup of $QI(\mathbb R)$ that fixes the ends $+\infty, -\infty$.   
We show that $QI^+(\mathbb R)\cong Q_+\times Q_-$. Using this we show that the center of the group $QI(\mathbb{R})$ is trivial.  
\end{abstract}

\section{Introduction}
We  begin by recalling the notion of quasi-isometry. 
 
Let $f:(X,d_X)\to (Y,d_Y)$ be a map between two metric spaces. We say that $f$ is a $K$-quasi-isometric embedding if there exists a $K>1$ such that
$$\frac{1}{K} d_X(x_1,x_2)-K\leq d_Y(f(x_1),f(x_2))\leq Kd_X(x_1,x_2)+K~ \forall x_1,x_2\in X.$$
 Again, if for any given point $y\in Y$, there is a point $x\in X$ such that $d_Y(f(x),y)<K,$ then $f$ is said to be a $K$-quasi-isometry. If $f$ is a quasi-isometry (for some $K>1$), then there exists a $K^\prime$- quasi-isometry $f^\prime: Y\to X$ such that $f^\prime\circ f$ (resp. $f\circ f^\prime$) is quasi-isometry, equivalent to the identity map of $X$ (resp. $Y$)(two maps $f,g:X\to Y$ are said to be quasi-isometrically equivalent if there exists a constant $M>0$ such that $d_Y(f(x),g(x))<M~\forall x\in X.$) Let $[f]$ denote the equivalence class of a quasi-isometry $f:X\to X.$ We denote the set of all equivalence classes of self-quasi-isometries of $X$ by $QI(X).$ 
 
It turns out that one has a well-defined notion of composition of isometry classes, where $[f]. [g]:=[f\circ g]$ for $[f], [g]\in QI(X).$
This makes $QI(\mathbb R)$ a group, referred to as the group of quasi-isometries of $(X,d_X)$.  
 If $\phi: X^\prime\to X$, $\phi':X\to X'$ are a pair of inverse quasi-isometries,  then $[f]\mapsto [\phi\circ f\circ \phi']$ defines an isomorphism 
 of groups $QI(X^\prime)\to QI(X).$    For example,  $t\to[t]$ is a quasi-isometry $\mathbb{R}\to \mathbb{Z},$ which induces 
 an isomorphism $QI(\mathbb{R})\cong QI(\mathbb Z)$.
  We refer the reader to \cite[Chapter I.8]{brid} for basic facts concerning quasi-isometry.

It is known that $QI(\mathbb R)$ is a rather large group; see \cite[3.3.B]{gromov-pansu}.    It is also known that the following groups can be embedded in $QI(\mathbb R)$ :-
 (i) $\widetilde{Diff}(\mathbb{S}^1)$ and $\widetilde{PL}(\mathbb{S}^1),$ (ii) $PL_\kappa(\mathbb{R}),$ (iii) the Thompson's group $F,$ (iv) the free group of rank $c,$  the continuoum. For the definitions and the embeddings of the aforementioned groups, the reader is referred to \cite{sank}.
 The group $PL_\kappa(\mathbb{R})$ is simple (see \cite{brin} and Theorem 3.1 of \cite{eps}).  The group $\widetilde{Diff}(\mathbb{S}^1)$ contains a free group of rank the continuoum (see \cite{grab}).  The Thomposon's group $F$ has many remarkable properties and arises in many different contexts in several branches of mathematics. We list below some properties of $F$: (i) the commutator subgroup $[F,F]$ is a simple group, (ii) every proper quotient group of $F$ is abelian, (iii) $F$ does not contain a non-abelian free subgroup (see \cite{cfp}).
Thus  we see that the group $QI(\mathbb{R})$ has a rich collection of subgroups having remarkable properties.  

However, the lattice of  {\it normal} subgroups of $QI(\mathbb R)$ does not seem to have been  studied.  As a first step in that 
direction we prove the following , which is the main result of this note.  
 
 \begin{theorem}\label{center}
The center of the group $QI(\mathbb{R})$ is trivial.
\end{theorem}

Our proof uses of the description $QI(\mathbb R)$ as a quotient of 
a certain subgroup of the group of all PL-homoemorphisms of $\mathbb R$ due to Sankaran \cite{sank}.

\section{PL-homeomorphisms with bounded slopes}

Let $f:\mathbb{R}\to \mathbb{R}$ be any homeomorphism of $\mathbb{R}.$ We denote by $B(f)$ the set of break points of $f,$ i.e. points where $f$ fails to have derivatives and by $\Lambda(f)$ the set of slopes of $f$, i.e.
 $$\Lambda(f)=\{f^\prime(t):t\in \mathbb{R}-B(f)\}.$$
 Note that $B(f)\subset\mathbb{R}$ is discrete if $f$ is piecewise differentiable.\\
 \begin{definition}
 We say that a subset $\Lambda$ of $\mathbb{R}^*$ (the set of non-zero real numbers) is {\it bounded} if there exists a $M>1$ such that $M^{-1}<|\lambda|<M$ for all $\lambda\in \Lambda.$  
 \end{definition} 
We denote, the set of all piecewise linear homeomorphisms $f$ of $\mathbb{R}$ such that $\Lambda(f)$ is bounded, by $PL_\delta(\mathbb{R}).$ This forms a subgroup of the group $PL(\mathbb{R})$ of all piecewise linear homeomorphisms of $\mathbb{R}.$  The subgroup of $PL_{\delta}(\mathbb R)$ 
consisting of orientation preserving PL-homeomorphisms will be denoted $PL_{\delta}^+(\mathbb R)$.  We have $PL_{\delta}(\mathbb R)=
PL_{\delta}^+(\mathbb R)\ltimes \langle \rho\rangle$ where $\rho$ is the reflection of $\mathbb R$ about the origin.
The following theorem is due to Sankaran \cite{sank}.

\begin{theorem} \label{sankaran}
The natural homomorphism $\phi:PL_\delta(\mathbb{R})\to QI(\mathbb{R}),$ defined as $f\to[f]$ is surjective.
\end{theorem}

The kernel of $\phi$ in the above theorem equals the subgroup of all $f\in PL_\delta(\mathbb R)$ such that 
$||f-id||<\infty$. In particular $\ker(\phi)$ contains all translations.  It follows that the restriction of $\phi$ to the subgroup $PL_{\delta,0}
:=\{f\in PL_\delta(\mathbb R)\mid f(0)=0\}$ is surjective. 

\noindent{\bf Notations.} 
We shall denote by $PL_{\delta,0}(\mathbb R)$ by $P$ and $QI(\mathbb R)$ by $Q$. We denote by $P_+$ (resp. $P_-$) 
the subgroup of $P$ consisting of homeomorphisms which are identity near $-\infty$ (resp. $+\infty$).  Note that $P_\kappa:=P_+\cap P_-$ 
is the group of compactly supported homeomorphisms in $P$.  Similarly, $Q_+( \text{resp.}~ Q_-)$ denote the subgroups of $QI(\mathbb{R})$ consisting of elements which are identity near $-\infty$ and $+\infty$ respectively.  The subgroup of $P$ consisting of orientation preserving homeomorphisms 
will be denoted by $P^+$.   Similarly, $Q^+$ denotes the subgroup of $Q$ whose elements can be represented by orientation preserving 
homeomorphisms and we have $\phi(P^+)=Q^+$, $\phi(P_\pm)=Q_\pm$.  
 
We shall denote by the same symbol $\phi$ the restriction of $\phi: PL_{\delta}(\mathbb R)\to QI(\mathbb R)$ to $P$.
The group $P=PL_{\delta,0}$ contains no non-trivial translations and, as already noted, we have $P=P^+\ltimes \langle \rho \rangle$.
Similarly the group $Q=QI(\mathbb{R})$ is a semi-direct product 
$Q=Q^+\ltimes \langle [\rho]\rangle$.    
 Then $Q_+\cap Q_-=\phi(P_+\cap P_-)=\phi(P_\kappa)$ is trivial and so we have $Q^+=Q_+\times Q_-$.   Conjugation by $[\rho]$ interchanges $Q_+$ and $Q_-$. 
It follows that if $N\subset Q^+$ is normal, if and only if its projections $N_+, N_-$ to the factors $Q_+,Q_-$ are normal.  
Also $N$ is a normal subgroup of $Q$ if and only if $N$ is normalized by $Q^+$ and $[\rho]N[\rho]^{-1}=N$.    
\begin{remark}
It may be an interesting problem to decide whether the groups $Q_+$ or $Q_-$ are simple.
\end{remark}

\section{Center of the group $QI(\mathbb{R})$}
We begin by making a few preliminary observations concerning central elements in $Q=QI(\mathbb R)$. 
We first observe that any element $[f]\in Q$ in the center must fix the ends $+\infty, -\infty$.  Indeed we may assume that 
$f(0)=0$; that $f(t)>0$ if and only if $t<0$.  Suppose that $f$ is orientation reversing.  Let $h\in P=PL_{\delta,0}(\mathbb R)$ be the orientation preserving 
PL-homeomorphism of $\mathbb R$ which has exactly one break point at the origin and $h'(t)=2$ for all $t>0$ and $h'(t)=1$ for all $t<0$.  By a straightforward 
computation $h(f(t))-f(h(t))=2f(t)-f(t)=f(t)$ for $t<0$.  So $||h\circ f-f\circ h||$ is unbounded and we conclude that $f$ is not in the 
center. 
 
Any element $f\in P^+=PL_{\delta,0}^+(\mathbb R)$ may be decomposed as a composition 
$f=f_+\circ f_-$ with $f_+\in P_+, f_-\in P_-$, where 
 $f_+(t)=t, f_-(s)=s$ for $s\le 0\le t$.   It follows that 
$[f]=([f_+],[f_-])\in Q_+\times Q_-=Q^+$ 
is in the center if and only if $f_\pm$ is in the center of $Q_\pm$ and $[\rho \circ f_+ \circ \rho]=[f_-]$.  
We will show that the center of $Q_+$ is trivial.  We need the following lemma.
 
\begin{lemma}\label{subsequence}
Let $f\in P_+$ and $[f]\neq [id]$.  
Then there exists a strictly monotone divergent sequence $\{b_n\}$ of real numbers such that at least 
one of the following holds:\\
(1) $b_n\to+\infty, ~ b_{n+1}>3f(b_n)~\forall n$ and $f(b_n)-b_n\to+\infty,$\\
(2) $b_n\to+\infty, ~ b_{n+1}>3f^{-1}(b_n)~\forall n$ and $f^{-1}(b_n)-b_n\to\infty,$\\
\end{lemma}
\begin{proof}
Since $f(t)=t$ for $t\le n_1$ (for some $n_1<0$) and since $||f-id||$ is unbounded, we can find a strictly monotone divergent sequence $\{a_n\}$ of positive real numbers
 such that $|f(a_{n+1})-a_{n+1}|>|f(a_n)-a_n|~\forall n,~|f(a_n)-a_n|\to\infty.$    We obtain a subsequence of $\{a_n\}$ as follows: 
Set $a_{n_1}=a_1$. Choose $n_2\ge 2$ such that $f(a_{n_2})>\max\{a_1, 3f(a_1)\}$.  Having chosen $a_{n_j}, 1\le j\le k$, we choose $n_{k+1}>n_k$
such that $f(a_{n_{k+1}})>\max\{a_{n_k}, 3f(a_{n_k})\}$.  We set $c_k:=a_{n_k}$.  Since $|f(c_k)-c_k|\to \infty$ as $k\to \infty$, 
there has to be infinitely many values of $k$ for which $f(c_k)-c_k$ has the same sign.  If this sign is positive, we have a subsequence 
$\{b_k\}$ of $\{c_k\}$ which satisfies (1).  If this sign is negative, then we apply the above consideration to $f^{-1}\in P_+$ which 
yields a sequence $\{b_k\}$ satisfying (2). 
\end{proof}
 


The previous lemma will be the main tool to study the center of the group $QI(\mathbb{R}).$ 

\noindent 
{\it Proof of Theorem \ref{center}.}  As observed already, it suffices to show that the center of $Q_+$ is trivial.  If possible, suppose 
that $[f]\ne 1$ is in the center of $Q_+$ with $f\in P_+$.  Then $[f^{-1}]$ is also in the center and so, without loss of generality 
we may assume the existence of a sequence $\{b_n\}$ satisfying (1) of the above lemma.   One has a PL-homeomorphism 
such that 
(i) $g'(t)=t, t\le b_1$,\\
(ii)  $g(J_k)=J_k$ where $J_k:=[b_k,b_{k+1}]$, and, \\
(iii)  $g$ has exactly one break point at $f(b_k)\in J_k$ and $g(f(b_k))=(b_k+f(b_k))/2$.  
We claim that $g\in P_+$ and that $[g]$ and $[f]$ do not commute.


First, we compute the slopes of $g$.  Since $g(J_k)=J_k$ and since $g$ has exactly one break point in $J_k$ at $f(b_k)$, 
straightforward computation shows that $g'(t) = 1/2$ if $b_k<t<f(b_k)$, and, when $f(b_k)<t<b_{k+1}$, we have \\
$g'(t)=(b_{k+1}-(b_k+f(b_k)/2))/(b_{k+1}-f(b_k))\\=1+(f(b_k)-b_k)/2(b_{k+1}-f(b_k))\\<1+(f(b_k)-b_k)/2(2f(b_k))<1+1/4.$ 
Since $f(b_k)>b_k$ we also see that $g'(t)>1$.  So we conclude that for any $t$ at which $g$ is differentiable, we have 
$g'(t)\in [1/2, 5/4]$, showing that $g\in P_+$.  

Finally, $f(g(b_k))=f(b_k)$ whereas $gf(b_k)=(f(b_k)+b_k)/2$.  So $f(g(b_k))-g(f(b_k))=(f(b_k)-b_k)/2\to \infty$ as $k\to \infty$ 
and so $[f][g]\ne [g][f]$.  This completes the proof. \hfill $\Box$


\begin{thebibliography}{99}
\bibitem{brid} M.R. Bridson and A. Haefliger: Metric spaces of non-positive curvature, Grund. Math. Wiss., 319, (1999), Springer-Verlag, Berlin.

\bibitem{brin} M. Brin and C.C. Squier: Groups of piecewise linear homeomorphisms of the real line, \textit{Invent. Math.}, \textbf{79}, (1985), 485--498.

\bibitem{cfp} J.W. Cannon, W.J. Floyd and W.R. Parry: Introductory notes on Richard Thompson's
groups, \textit{Enseign. Math.} \textbf{42} 215--256, (1996).


\bibitem{eps} D. B. A. Epstein: The simplicity of certain groups of homeomorphisms, \textit{Compositio Math.}, \textbf{22} (1970) 165--173. 


\bibitem{gromov-pansu} M. Gromov and P. Pansu, Rigidity of lattices: An introduction. in {\it Geometry and Topology: Recent Developments,}  eds. P. de Bartolomeis, F. TricerriLect. Notes Math. {\bf 1504}, Springer-Verlag, Berlin, 1991.

\bibitem{grab} J. Grabowski: Free subgroups of diffeomrphism groups, \textit{Fund. Math.}, \textbf{131} (1988), 103--121.

\bibitem{sank} P. Sankaran: On homeomorphisms and quasi-isometries of the real line. \textit{Proc. of the Amer. Math. Soc.}, \textbf{134} (2005) 1875--1880.

\end{thebibliography}
\end{document}